\documentclass[11pt]{amsart}

\usepackage{amsmath,amssymb,amsthm,mathtools}
\usepackage{enumitem}
\usepackage[colorlinks=true,citecolor=blue,linkcolor=blue,urlcolor=blue]{hyperref}

\newtheorem{theorem}{Theorem}[section]
\newtheorem{proposition}[theorem]{Proposition}
\newtheorem{lemma}[theorem]{Lemma}

\theoremstyle{definition}

\theoremstyle{remark}

\newcommand{\R}{\mathbb R}

\newcommand{\Euc}{\mathrm{Euc}}

\title[STABLE CONSTANT WEIGHTED MEAN CURVATURE HYPERSURFACES]{STABLE CONSTANT WEIGHTED MEAN CURVATURE\\ HYPERSURFACES IN ANTI-GAUSSIAN SPACE}

\author[Bai]{Jinchuan Bai}
\address[J.B.]{School of Mathematical Sciences\\
Xiamen University\\
361005\\
Xiamen\\
P. R.  China}
\email{baijinchuan@stu.xmu.edu.cn
}

\author[Xia]{Chao Xia}
\address[C.X.]{School of Mathematical Sciences\\
Xiamen University\\
361005\\
Xiamen\\
P. R.  China}
\email{chaoxia@xmu.edu.cn}
\thanks{This work is supported by
NSFC (Grant No.  12271449, 12526203, 12526102) and the Natural Science Foundation of
Fujian Province of China (Grant No.  2024J011008).}

\begin{document}

\begin{abstract}
We prove that in the anti-Gaussian space $\bigl(\R^{m+1},\bar g_{\Euc},e^{|x|^2/4}\,dx\bigr)$, $m\ge2$, round spheres centered at the origin are the only closed, connected, two-sided immersed hypersurfaces with constant weighted mean curvature that are stable under weighted-volume-preserving variations.
\end{abstract}

\maketitle

\section{Introduction}
Isoperimetric problem asks for regions of least perimeter enclosing a fixed volume. It is well-known that isoperimetric regions in the Euclidean space are round balls.  Isoperimetric problems have also been studied extensively in the setting of manifolds with density
$
\bigl(M^{m+1},\bar g,e^{-\phi}d\mu_{\bar g}\bigr),$ where $\phi$ is a smooth function on $M$ and $\mu_{\bar g}$ is the Riemannian volume form.
It seeks regions $\Omega$ of least weighted perimeter $\int_{\partial \Omega} e^{-\phi}d\sigma_{\bar g}$ enclosing a fixed weighted volume $\int_{\Omega} e^{-\phi}d\mu_{\bar g}$.

One of the first and most interesting examples, with applications in probability and statistics, is the Gaussian space $$
\bigl(\mathbb{R}^{m+1},\bar g_{Euc},e^{-\frac{|x|^2}{4}}dx\bigr).$$ Sudakov-Tirel'son \cite{ST} and Borell \cite{Borell} independently proved that the Gaussian isoperimetric regions are given by half-spaces. In a recent joint work of the authors with Lu \cite{BaiLuXia}, the
capillary isoperimetric problem in a Gaussian half-space was studied,
and the minimizers were shown to be the half-spaces meeting the supporting hyperplane at the prescribed contact angle.
Rosales-Canete-Bayle-Morgan \cite{RosalesCaneteBayleMorgan} have studied Euclidean space with general density.
Among other results, they showed that the isoperimetric regions in the anti-Gaussian space
 $$\bigl(\mathbb{R}^{m+1},\bar g_{Euc},e^{\frac{|x|^2}{4}}dx\bigr)$$ are given by balls centered at the origin. Later, the isoperimetric problem for general radial log-convex densities, due to Brakke, has been solved by Chambers \cite{Chambers}.

It is also interested to consider stable critical points or local minimizers for isoperimetric problems. 
Barbosa-do Carmo \cite{BarbosaDoCarmo} showed that the only stable closed solutions to isoperimetric problems in the Euclidean space $\mathbb{R}^{m+1}$ are round spheres. Combined with the stability criteria by Sternberg-Zumbrun \cite{SternbergZumbrun}, it shows that local minimizers are round spheres.
McGonagle-Ross \cite{McGonagleRoss} showed that the only stable smooth complete solutions to isoperimetric problems in the Gaussian space are hyperplanes. 

In this paper, we consider stable critical points to isoperimetric problems in the anti-Gaussian space.
 The critical points for isoperimetric problems in the anti-Gaussian space are given by constant weighted mean curvature hypersurfaces, that is
\[
H+\frac12\langle x,\nu\rangle=\lambda, \quad \lambda\in \mathbb{R}.
\]
When $\lambda=0$, it is so-called self-expander to mean curvature flow, which has been well studied, see e.g. \cite{BernsteinWang,DeruelleSchulze,Ding}. For general $\lambda$, it is also referred to as $\lambda$-self-expander, see e.g. Ancari-Cheng \cite{AncariCheng}. An Alexandrov-type theorem, saying that an closed embedded
$\lambda$-self-expander must be a sphere centered at the origin, has been proved independently by the authors \cite{BaiXia} and Johne-Silini \cite{JohneSilini}.

A constant weighted mean curvature hypersurface is called stable if the second variation of weighted perimeter is nonnegative among weighted volume preserving variations. We recall the first- and second-variation formulas in Section \ref{sec:variational}.
Rosales-Canete-Bayle-Morgan \cite{RosalesCaneteBayleMorgan} proved that a sphere centered at the origin is stable. 
The main result is the following.

\begin{theorem}\label{thm:main}
A stable, smooth, closed, connected, two-sided immersed constant weighted mean curvature hypersurface in  the anti-Gaussian space is a round sphere centered at the origin. As a consequence, 
local minimizers of the isoperimetric problems in the anti-Gaussian space are  round balls centered at the origin.
\end{theorem}
We illustrate the idea of the proof. 

For a closed constant weighted mean curvature immersed hypersurface $\Sigma$,
the stability yields that
$$Q(\psi,\psi)=-\int_\Sigma \psi\left[\Delta_\phi \psi+(|A|^2-\frac12)\psi\right]d\sigma_\phi\ge 0,\quad \int_\Sigma \psi d\sigma_\phi= 0,$$
where $\Delta_\phi=\Delta+\frac12\left<x,\nabla \right>$.
To prove the theorem, as in \cite{BarbosaDoCarmo, BarbosaDoCarmoEschenburg, WangXia}, we use the Minkowski formula
 $$\int_\Sigma (V-H_\phi \left<x, \nu\right>)d\sigma_\phi=0,$$
 where $V=m+\frac12|x|^2$ and $H_\phi=H+\frac12\left<x,\nu\right>.$  Thus $\psi=V-H_\phi \left<x, \nu\right>$ is an admissible function in the stability inequality. By using $\psi$, we get
 $$Q(\psi,\psi)=-\int_\Sigma (V-H_\phi \left<x, \nu\right>)\left[(|A|^2+\frac12)V-H_\phi^2\right]d\sigma_\phi\ge 0.$$
It is known by the authors \cite{BaiXia} that 
$$(|A|^2+\frac12)V-H_\phi^2\ge 0.$$
However, there is in general no sign for $V-H_\phi \left<x, \nu\right>$.
In previous papers \cite{BarbosaDoCarmo, BarbosaDoCarmoEschenburg, WangXia}, one usually added some $\delta_\S \Phi$ to stability inequality which has vanishing integral but can compensate pointwisely good terms. In this paper, we use a new way to handle the problem.
Our key observation is that 
 $$\int_\Sigma H_\phi \left<x, \nu\right>\left[(|A|^2+\frac12)V-H_\phi^2\right]d\sigma_\phi=\int_\Sigma H_\phi A(x^T, x^T)d\sigma_\phi,$$
and a key pointwise algebraic lemma tells that $$V\left[(|A|^2+\frac12)V-H_\phi^2\right]-H_\phi A(x^T, x^T)\ge 0.$$
This yields the proof for smooth hypersurfaces.
Combined with the stablity criteria by Sternberg-Zumbrun \cite{SternbergZumbrun} and the regularity theory for local minimizers \cite{GMT, MorganRegularity}, the same proof works for local minimizers for isoperimetric problems.

Related results on stability and isoperimetric problems can be found in
\cite{Alexandrov,BarbosaDoCarmoEschenburg,Montiel,RitoreRos,MorganRos}.
For manifolds with density, related isoperimetric and stability problems have been studied in
\cite{FigalliMaggi,MorganDensity,RosalesStableSets}.
The free boundary and capillary cases have also been considered in
\cite{BaiLuXia, CastroRosales,GuoWangXia, Souam,WangXia}.

The paper is organized as follows. In Section \ref{sec:preliminaries}, we recall the variational formulas, the weighted Minkowski formula, and the basic differential identities in the anti-Gaussian space. In Section \ref{sec:proof}, we prove Theorem \ref{thm:main}. A short compactness lemma for local minimizers is given in Appendix \ref{app:localmin}.

\section{Preliminaries for hypersurfaces in the anti-Gaussian space}\label{sec:preliminaries}\label{sec:variational}

Let $m\ge 2$ and $x:\Sigma^m\to\mathbb R^{m+1}$ be a closed, two-sided immersed hypersurface with globally defined unit normal $\nu$. We use the convention
\[
A(X,Y)=\langle\overline{\nabla}_X\nu,Y\rangle,
\qquad H=\operatorname{tr}_\Sigma A,
\]
where $A$ and $H$ are the second fundamental form and the mean curvature of $\Sigma$, respectively.

Throughout this section the ambient space is the anti-Gaussian space
\[
\bigl(\mathbb R^{m+1},\bar g_{\mathrm{Euc}},e^{|x|^2/4}\,dx\bigr).
\]

Thus
\[
\phi=-\frac{|x|^2}{4},
\qquad
d\sigma_\phi=e^{|x|^2/4}\,d\sigma,
\qquad
H_\phi=H+\frac12\langle x,\nu\rangle.
\]
The weighted Laplacian and the Jacobi operator are
\begin{equation}\label{eq:operators}
\Delta_\phi f
=\Delta f+\frac12\langle x^\top,\nabla f\rangle,
\qquad
L_\phi f
=\Delta_\phi f+\left(|A|^2-\frac12\right)f.
\end{equation}

The standard first and second variation formulas for hypersurfaces with density can be found, for example, in \cite{CastroRosales,McGonagleRoss,RosalesCaneteBayleMorgan}. Let $x_t$ be a smooth normal variation of $x$ with normal speed
\[
\psi=\left\langle \left.\frac{\partial x_t}{\partial t}\right|_{t=0},\nu\right\rangle.
\]
Since $\Sigma$ is only assumed to be immersed, by weighted volume we
mean the signed weighted volume associated with the variation.
A variation is weighted-volume-preserving if this signed volume is
constant; in particular,
\begin{equation}\label{eq:mean-zero}
\int_\Sigma \psi\,d\sigma_\phi=0.
\end{equation}
Conversely, every smooth function satisfying \eqref{eq:mean-zero} can be realized as the normal speed of a weighted-volume-preserving variation, see e.g. \cite{BarbosaDoCarmoEschenburg}. The first variation of the weighted area is
\begin{equation}\label{eq:first-variation}
\left.\frac{d}{dt}\right|_{t=0}\operatorname{Area}_\phi(\Sigma_t)
=\int_\Sigma H_\phi\psi\,d\sigma_\phi.
\end{equation}
Hence a hypersurface is a critical point of the weighted area under the weighted-volume constraint if and only if $H_\phi$ is constant. For such a hypersurface, the second variation is
\begin{equation}\label{eq:second-variation}
Q(\psi,\psi)
=-\int_\Sigma \psi L_\phi\psi\,d\sigma_\phi
=\int_\Sigma\left(|\nabla\psi|^2-\left(|A|^2-\frac12\right)\psi^2\right)d\sigma_\phi.
\end{equation}
Thus stability means
\begin{equation}\label{eq:stability}
Q(\psi,\psi)\ge0
\qquad\text{for every }\psi\in C^\infty(\Sigma)
\text{ satisfying }\int_\Sigma\psi\,d\sigma_\phi=0.
\end{equation}

We next recall the weighted Minkowski formula and several elementary identities for constant weighted mean curvature hypersurfaces. These formulas are standard in the study of $\lambda$-self-expanders; see, for example, \cite{AncariCheng,BaiXia, JohneSilini}. For completeness, we include the short computations.

Set
\begin{equation}\label{eq:basic-definitions}
V:=m+\frac{|x|^2}{2},
\qquad
P:=|A|^2+\frac12,
\qquad
F:=V-H_\phi \langle x,\nu\rangle.
\end{equation}
Then $L_\phi=\Delta_\phi+P-1$.

\begin{lemma}[Weighted Minkowski formula]\label{lem:Minkowski}
Let $\Sigma$ be a closed two-sided immersed hypersurface with constant $H_\phi$. Then
\begin{equation}\label{eq:Minkowski}
\int_\Sigma \bigl(V-H_\phi \langle x,\nu\rangle\bigr)\,d\sigma_\phi=0.
\end{equation}
\end{lemma}

\begin{proof}
For a tangent vector field $Y$,
\[
\operatorname{div}_{\Sigma,\phi}Y
=\operatorname{div}_\Sigma Y+\frac12\langle x^\top,Y\rangle.
\]
Since
\[
\operatorname{div}_\Sigma x^\top=m-H\langle x,\nu\rangle,
\]
we obtain
\[
\operatorname{div}_{\Sigma,\phi}x^\top
=m-H\langle x,\nu\rangle+\frac12|x^\top|^2
=V-H_\phi \langle x,\nu\rangle.
\]
Integrating over the closed hypersurface gives \eqref{eq:Minkowski}.
\end{proof}

\begin{lemma}[Basic identities]\label{lem:differential}
Assume that $H_\phi$ is constant. Then
\begin{align}
\nabla V&=x^\top, \label{eq:grad-identities}\\
\Delta_\phi V&=V-H_\phi \langle x,\nu\rangle, \label{eq:DeltaV}\\
\Delta_\phi \langle x,\nu\rangle&=H_\phi-P\langle x,\nu\rangle, \label{eq:Deltau}\\
L_\phi V&=PV-H_\phi \langle x,\nu\rangle, \label{eq:LV}\\
L_\phi \langle x,\nu\rangle&=H_\phi-\langle x,\nu\rangle, \label{eq:Lu}\\
L_\phi F&=PV-H_\phi^2. \label{eq:LF}
\end{align}
\end{lemma}

\begin{proof}
The identity in \eqref{eq:grad-identities} follow directly from the definitions. Moreover,
\[
\Delta\frac{|x|^2}{2}=m-H\langle x,\nu\rangle,
\]
and therefore
\[
\Delta_\phi V
=m-H\langle x,\nu\rangle+\frac12|x^\top|^2
=V-H_\phi \langle x,\nu\rangle,
\]
which proves \eqref{eq:DeltaV}.

For the support function $\langle x,\nu\rangle$, the Euclidean formula is
\[
\Delta \langle x,\nu\rangle=H-|A|^2 \langle x,\nu\rangle+\langle x^\top,\nabla H\rangle.
\]

Using \eqref{eq:operators}, we obtain
\[
\Delta_\phi \langle x,\nu\rangle
=H-|A|^2 \langle x,\nu\rangle
=H_\phi-\left(|A|^2+\frac12\right)\langle x,\nu\rangle,
\]
which is \eqref{eq:Deltau}. Equations \eqref{eq:LV} and \eqref{eq:Lu} follow immediately from $L_\phi=\Delta_\phi+P-1$. Finally, since $H_\phi$ is constant,
\[
L_\phi F
=L_\phi V-H_\phi L_\phi \langle x,\nu\rangle
=PV-H_\phi^2,
\]
which proves \eqref{eq:LF}.
\end{proof}

\section{Proof of the main theorem}\label{sec:proof}

By Lemma \ref{lem:Minkowski}, the function $F$ in \eqref{eq:basic-definitions} satisfies
\[
\int_\Sigma F\,d\sigma_\phi=0,
\]
and is therefore admissible in the stability inequality \eqref{eq:stability}.

We first note a pointwise result from Cauchy-Schwarz.

\begin{proposition}\label{lem:positivity}
On $\Sigma$,
\begin{equation}\label{eq:LF-positive}
L_\phi F=PV-H_\phi^2\ge\frac12P|x^\top|^2\ge0.
\end{equation}
\end{proposition}

\begin{proof}
Write
\[
V=m+\frac{\langle x,\nu\rangle^2}{2}+\frac{|x^\top|^2}{2}.
\]
At a fixed point choose an orthonormal principal frame with principal curvatures $\kappa_1,\ldots,\kappa_m$. Applying Cauchy--Schwarz to
\[
a=\left(\kappa_1,\ldots,\kappa_m,\frac1{\sqrt2}\right),
\qquad
b=\left(1,\ldots,1,\frac{\langle x,\nu\rangle}{\sqrt2}\right),
\]
we obtain
\[
H_\phi^2
=\left(H+\frac12 \langle x,\nu\rangle\right)^2
\le\left(|A|^2+\frac12\right)
\left(m+\frac{\langle x,\nu\rangle^2}{2}\right).
\]
Hence
\[
PV-H_\phi^2\ge\frac12P|x^\top|^2.
\]
\end{proof}

\begin{proposition}\label{prop:integral-identity}
Assume that $H_\phi$ is constant. Then
\[
Q(F,F)
=\int_\Sigma\Bigl[H_\phi A(x^\top,x^\top)-VL_\phi F\Bigr]d\sigma_\phi.
\]
\end{proposition}

\begin{proof}
From
\[
\Delta_\phi\langle x,\nu\rangle
=
H_\phi-P\langle x,\nu\rangle
\]
and $\nabla V=x^\top$, integration by parts gives
\[
\begin{aligned}
\int_\Sigma
V\bigl(H_\phi-P\langle x,\nu\rangle\bigr)\,d\sigma_\phi
&=
\int_\Sigma
V\Delta_\phi\langle x,\nu\rangle\,d\sigma_\phi\\
&=
-\int_\Sigma
\left\langle\nabla V,
\nabla\langle x,\nu\rangle\right\rangle\,d\sigma_\phi\\
&=
-\int_\Sigma
A(x^\top,x^\top)\,d\sigma_\phi.
\end{aligned}
\]
Hence
\[
\int_\Sigma
PV\langle x,\nu\rangle\,d\sigma_\phi
=
H_\phi\int_\Sigma V\,d\sigma_\phi
+
\int_\Sigma A(x^\top,x^\top)\,d\sigma_\phi.
\]

By the Lemma \ref{lem:Minkowski},
we obtain
\[
\begin{aligned}
\int_\Sigma
\langle x,\nu\rangle L_\phi F\,d\sigma_\phi
&=
\int_\Sigma
PV\langle x,\nu\rangle\,d\sigma_\phi
-
H_\phi^2
\int_\Sigma\langle x,\nu\rangle\,d\sigma_\phi\\
&=
\int_\Sigma
A(x^\top,x^\top)\,d\sigma_\phi.
\end{aligned}
\]

So,
\[
\begin{aligned}
Q(F,F)&=-\int_\Sigma V L_\phi F\,d\sigma_\phi
+H_\phi\int_\Sigma \langle x,\nu\rangle L_\phi F\,d\sigma_\phi\\
&=\int_\Sigma\Bigl[H_\phi A(x^\top,x^\top)-VL_\phi F\Bigr]d\sigma_\phi.
\end{aligned}
\]
\end{proof}

The following pointwise estimate is the key step.

\begin{proposition}\label{prop:key}
On $\Sigma$,
\begin{equation}\label{eq:key}
VL_\phi F=V(PV-H_\phi^2)\ge H_\phi A(x^\top,x^\top).
\end{equation}
Moreover, the inequality is strict whenever $x^\top\neq0$.
\end{proposition}

\begin{proof}
If $x^\top=0$, then $A(x^\top,x^\top)=0$, and \eqref{eq:key} follows from Proposition \ref{lem:positivity}. Assume $x^\top\neq0$, and choose an orthonormal tangent frame $\{e_i\}_{i=1}^m$ with
\[
e_1=\frac{x^\top}{|x^\top|},
\qquad h_{ij}=A(e_i,e_j).
\]
Then $A(x^\top,x^\top)=|x^\top|^2h_{11}$. Since $2H+\langle x,\nu\rangle=2H_\phi$,
\[
\begin{aligned}
VL_\phi F
&=V^2P-VH_\phi^2\\
&=\sum_{i,j=1}^m(Vh_{ij}-H_\phi\delta_{ij})^2
+\frac12(V-\langle x,\nu\rangle H_\phi)^2
+\frac{|x^\top|^2}{2}H_\phi^2.
\end{aligned}
\]
Therefore
\[
VL_\phi F\ge(Vh_{11}-H_\phi)^2+\frac{|x^\top|^2}{2}H_\phi^2.
\]
It follows that
\[
\begin{aligned}
VL_\phi F-H_\phi A(x^\top,x^\top)
&\ge(Vh_{11}-H_\phi)^2
+\frac{|x^\top|^2}{2}H_\phi^2
-|x^\top|^2H_\phi h_{11}\\
&=\left[Vh_{11}-\frac{V+|x^\top|^2/2}{V}H_\phi\right]^2+\frac{|x^\top|^2}{2V^2}
\left(V^2-2V-\frac{|x^\top|^2}{2}\right)H_\phi^2.
\end{aligned}
\]
Since $m\ge2$,
\[
V=m+\frac{\langle x,\nu\rangle^2}{2}+\frac{|x^\top|^2}{2}
\ge2+\frac{|x^\top|^2}{2},
\]
and hence
\[
V^2-2V-\frac{|x^\top|^2}{2}
\ge\frac{|x^\top|^2}{2}(V-1)\ge0.
\]
This proves \eqref{eq:key}. If $x^\top\neq0$, then the last term is strictly positive because $V>1$ and $H_\phi\neq0$ by Lemma \ref{lem:Minkowski}.
\end{proof}

\begin{proof}[Proof of Theorem \ref{thm:main}]
By stability,
\[
Q(F,F)\ge0.
\]
On the other hand, Proposition \ref{prop:integral-identity} and Proposition \ref{prop:key} give
\[
Q(F,F)\le0.
\]
Thus $Q(F,F)=0$, and
\[
\int_\Sigma\Bigl[VL_\phi F-H_\phi A(x^\top,x^\top)\Bigr]d\sigma_\phi=0.
\]
The integrand is nonnegative and, by the strict part of Proposition \ref{prop:key}, it is positive wherever $x^\top\neq0$. Therefore
\[
x^\top\equiv0.
\]
Consequently $\nabla|x|^2=2x^\top=0$, so $|x|\equiv R$ for some $R>0$. Hence the image of $x$ is contained in $S_R^m(0)$. Since $x:\Sigma\to S_R^m(0)$ is an immersion of the same dimension and $\Sigma$ is compact and connected, its image is both open and closed in $S_R^m(0)$. Moreover, $x:\Sigma\to S_R^m(0)$ is a covering map.
Since $m\ge2$, the sphere $S_R^m(0)$ is simply connected.
Hence $x$ is a diffeomorphism. Therefore
\[
\Sigma=S_R^m(0).
\]

It remains to justify the consequence for local minimizers. Let $E$ be a
finite-perimeter weighted $L^1$-local minimizer with
\[
0<V_\phi(E)<\infty,
\qquad
P_\phi(E)<\infty.
\]
Here
\[
V_\phi(E):=\int_E e^{|x|^2/4}\,dx,
\qquad
P_\phi(E;U):=
\int_{\partial^*E\cap U}e^{|x|^2/4}\,d\mathcal H^m,
\]
and $P_\phi(E):=P_\phi(E;\mathbb R^{m+1})$.

By Lemma \ref{lem:compact-localmin}, the support of the perimeter measure
is compact. The regularity theory for volume-constrained perimeter
minimizers \cite{MorganRegularity,GMT,RosalesCaneteBayleMorgan} gives
\[
\operatorname{spt}|D\chi_E|=\Sigma\cup\mathcal S,
\qquad
\dim_{\mathcal H}\mathcal S\le m-7,
\]
where $\Sigma$ is the smooth embedded regular part. The first variation
gives one constant value of $H_\phi$ on all regular components, and local
minimality gives weighted-volume-preserving stability.

Since the boundary is compact, the density and its inverse are uniformly
bounded near $\operatorname{spt}|D\chi_E|$. Thus weighted and unweighted
$W^{1,2}$ capacities are equivalent there. The cutoff argument of
Sternberg--Zumbrun \cite{SternbergZumbrun}, together with its weighted
version in \cite{RosalesStableSets}, applies across $\mathcal S$. By the
standard approximation argument used in these references, the Minkowski
identity, the stability inequality, and the integrations by parts above
extend to the regular part. Hence
\begin{equation}\label{eq:localmin-three-identities}
\int_\Sigma F\,d\sigma_\phi=0,
\qquad
Q(F,F)\ge0,
\qquad
Q(F,F)=
\int_\Sigma
\Bigl[H_\phi A(x^\top,x^\top)-VL_\phi F\Bigr]d\sigma_\phi.
\end{equation}
Proposition \ref{prop:key} therefore gives
\[
0\le Q(F,F)\le0,
\]
and the strict part of the proposition yields
\[
x^\top=0
\qquad\text{on }\Sigma.
\]

Thus every regular component is contained in a centered sphere. Since
$H_\phi$ has the same constant value $\lambda$ on all components, every
possible radius $R$ satisfies
\[
\frac{m}{R}+\frac{R}{2}=|\lambda|,
\]
so only finitely many radii can occur. Applying the constancy theorem
\cite{GMT} to the integral boundary current
$\partial[[E]]$ in annular neighborhoods separating these radii shows
that each nonzero component is a whole centered sphere with multiplicity
one. If two distinct concentric components were present, two successive
components would have opposite orientations and hence opposite signs of
$H_\phi$, contradicting the constancy of $H_\phi$. Therefore
$\partial^*E$ consists, up to an $\mathcal H^m$-null set, of a single
sphere $S_R^m(0)$.

It follows that $E$ agrees, up to a null set, with either $B_R(0)$ or its
complement. Since the complement has infinite anti-Gaussian weighted
volume,
\[
E=B_R(0)
\]
up to a null set. This proves the statement for local minimizers.
\end{proof}

\appendix
\section{Compactness of local minimizers}\label{app:localmin}

The regularity and cutoff arguments used above are standard and are
covered by \cite{MorganRegularity,GMT,SternbergZumbrun,RosalesStableSets}.
We only record the compactness fact that uses the growth of the
anti-Gaussian density.

\begin{lemma}\label{lem:compact-localmin}
Let $E\subset\mathbb R^{m+1}$ be a finite-perimeter weighted
$L^1$-local minimizer with
\[
0<V_\phi(E)<\infty,
\qquad
P_\phi(E)<\infty.
\]
Then $\operatorname{spt}|D\chi_E|$ is compact.
\end{lemma}

\begin{proof}
By the standard volume-fixing argument for volume-constrained perimeter
minimizers
\cite{MorganRegularity,GMT,RosalesCaneteBayleMorgan},
$E$ is an almost-minimizer of the weighted perimeter away from a fixed
compact set. Suppose that
$p_j\in\operatorname{spt}|D\chi_E|$ and $|p_j|\to\infty$, and set
\[
r_j=
\varepsilon
\exp\left(-\frac{|p_j|^2}{4(m+1)}\right).
\]
Since $|p_j|r_j\to0$, the normalized densities
\[
\frac{e^{|p_j+r_jy|^2/4}}{e^{|p_j|^2/4}}
\]
have uniform ellipticity and $C^1$ bounds on $B_2(0)$. Moreover,
\[
V_\phi(B_{r_j}(p_j))
\le C e^{|p_j|^2/4}r_j^{m+1}
=C\varepsilon^{m+1},
\]
so $\varepsilon$ can be chosen so that the local minimizing property
applies in these balls. The density theorem for local minimizers of
area-type functionals \cite{CongedoTamanini} then gives a constant
$c>0$, independent of $j$, such that
\[
P_\phi(E;B_{r_j}(p_j))
\ge
c\,e^{|p_j|^2/4}r_j^m
=
c\varepsilon^m
\exp\left(\frac{|p_j|^2}{4(m+1)}\right).
\]
The right-hand side tends to infinity, contradicting
$P_\phi(E)<\infty$. Hence
$\operatorname{spt}|D\chi_E|$ is bounded, and therefore compact.
\end{proof}

\bibliographystyle{plain}
\bibliography{References.bib}

@article{Alexandrov,
  author  = {A. D. Alexandrov},
  title   = {Uniqueness theorems for surfaces in the large. {I}},
  journal = {Vestnik Leningrad University},
  volume  = {11},
  number  = {19},
  pages   = {5--17},
  year    = {1956}
}

@article{AncariCheng,
  author  = {S. Ancari and X. Cheng},
  title   = {Some rigidity properties for $\lambda$-self-expanders},
  journal = {Nonlinear Anal.},
  volume  = {230},
  pages   = {113230},
  year    = {2023}
}

@article{BaiXia,
  author  = {J. Bai and C. Xia},
  title   = {An Alexandrov-type theorem in warped product manifolds with radial density},
  journal = {arXiv:2608.08548},
  year    = {2026}
}

@article{BarbosaDoCarmo,
  author  = {J. L. Barbosa and M. do Carmo},
  title   = {Stability of hypersurfaces with constant mean curvature},
  journal = {Math. Z.},
  volume  = {185},
  number  = {3},
  pages   = {339--353},
  year    = {1984}
}

@article{BarbosaDoCarmoEschenburg,
  author  = {J. L. Barbosa and M. do Carmo and J. Eschenburg},
  title   = {Stability of hypersurfaces of constant mean curvature in Riemannian manifolds},
  journal = {Math. Z.},
  volume  = {197},
  number  = {1},
  pages   = {123--138},
  year    = {1988}
}

@article{BernsteinWang,
  author  = {J. Bernstein and L. Wang},
  title   = {The space of asymptotically conical self-expanders of mean curvature flow},
  journal = {Math. Ann.},
  volume  = {380},
  number  = {1},
  pages   = {175--230},
  year    = {2021}
}

@article{CastroRosales,
  author  = {K. Castro and C. Rosales},
  title   = {Free boundary stable hypersurfaces in manifolds with density and rigidity results},
  journal = {J. Geom. Phys.},
  volume  = {79},
  pages   = {14--28},
  year    = {2014}
}

@article{Chambers,
  author  = {G. R. Chambers},
  title   = {Proof of the log-convex density conjecture},
  journal = {J. Eur. Math. Soc.},
  volume  = {21},
  number  = {8},
  pages   = {2301--2332},
  year    = {2019}
}

@article{DeruelleSchulze,
  author  = {A. Deruelle and F. Schulze},
  title   = {Generic uniqueness of expanders with vanishing relative entropy},
  journal = {Math. Ann.},
  volume  = {377},
  number  = {3},
  pages   = {1095--1127},
  year    = {2020}
}

@article{Ding,
  author  = {Q. Ding},
  title   = {Minimal cones and self-expanding solutions for mean curvature flows},
  journal = {Math. Ann.},
  volume  = {376},
  number  = {1},
  pages   = {359--405},
  year    = {2020}
}

@article{FigalliMaggi,
  author  = {A. Figalli and F. Maggi},
  title   = {On the isoperimetric problem for radial log-convex densities},
  journal = {Calc. Var. Partial Differential Equations},
  volume  = {48},
  number  = {3--4},
  pages   = {447--489},
  year    = {2013}
}

@article{GuoWangXia,
  author  = {J. Guo and G. Wang and C. Xia},
  title   = {Stable capillary hypersurfaces supported on a horosphere in the hyperbolic space},
  journal = {Adv. Math.},
  volume  = {409},
  pages   = {108641},
  year    = {2022}
}

@article{JohneSilini,
  author  = {F. Johne and L. Silini},
  title   = {Alexandrov theorem for constant weighted mean curvature surfaces},
  journal = {arXiv:2608.05119},
  year    = {2026}
}

@article{McGonagleRoss,
  author  = {M. McGonagle and J. Ross},
  title   = {The hyperplane is the only stable, smooth solution to the isoperimetric problem in Gaussian space},
  journal = {Geom. Dedicata},
  volume  = {178},
  number  = {1},
  pages   = {277--296},
  year    = {2015}
}

@article{Montiel,
  author  = {S. Montiel},
  title   = {Stable constant mean curvature hypersurfaces in some Riemannian manifolds},
  journal = {Comment. Math. Helv.},
  volume  = {73},
  number  = {4},
  pages   = {584--602},
  year    = {1998}
}

@article{MorganRegularity,
  author  = {F. Morgan},
  title   = {Regularity of isoperimetric hypersurfaces in Riemannian manifolds},
  journal = {Trans. Amer. Math. Soc.},
  volume  = {355},
  number  = {12},
  pages   = {5041--5052},
  year    = {2003}
}

@book{GMT,
  author    = {F. Morgan},
  title     = {Geometric Measure Theory: A Beginner's Guide},
  edition   = {4},
  publisher = {Elsevier/Academic Press},
  address   = {Amsterdam},
  year      = {2009}
}

@article{SternbergZumbrun,
  author  = {P. Sternberg and K. Zumbrun},
  title   = {A Poincar{\'e} inequality with applications to volume-constrained area-minimizing surfaces},
  journal = {J. Reine Angew. Math.},
  volume  = {503},
  pages   = {63--85},
  year    = {1998}
}

@article{MorganDensity,
  author  = {F. Morgan},
  title   = {Manifolds with density},
  journal = {Notices Amer. Math. Soc.},
  volume  = {52},
  number  = {8},
  pages   = {853--858},
  year    = {2005}
}

@article{MorganRos,
  author  = {F. Morgan and A. Ros},
  title   = {Stable constant-mean-curvature hypersurfaces are area minimizing in small $L^1$ neighborhoods},
  journal = {Interfaces Free Bound.},
  volume  = {12},
  number  = {2},
  pages   = {151--155},
  year    = {2010}
}

@article{RitoreRos,
  author  = {M. Ritor{\'e} and A. Ros},
  title   = {Stable constant mean curvature tori and the isoperimetric problem in three space forms},
  journal = {Comment. Math. Helv.},
  volume  = {67},
  number  = {2},
  pages   = {293--305},
  year    = {1992}
}

@article{RosalesStableSets,
  author  = {C. Rosales},
  title   = {Isoperimetric and stable sets for log-concave perturbations of Gaussian measures},
  journal = {Anal. Geom. Metr. Spaces},
  volume  = {2},
  pages   = {328--358},
  year    = {2014}
}

@article{RosalesCaneteBayleMorgan,
  author  = {C. Rosales and A. Ca{\~n}ete and V. Bayle and F. Morgan},
  title   = {On the isoperimetric problem in Euclidean space with density},
  journal = {Calc. Var. Partial Differential Equations},
  volume  = {31},
  number  = {1},
  pages   = {27--46},
  year    = {2008}
}

@article{WangXia,
  author  = {G. Wang and C. Xia},
  title   = {Uniqueness of stable capillary hypersurfaces in a ball},
  journal = {Math. Ann.},
  volume  = {374},
  number  = {3--4},
  pages   = {1845--1882},
  year    = {2019}
}

@article{ST,
  author  = {V. N. Sudakov and B. S. Tsirel'son},
  title   = {Extremal properties of half-spaces for spherically invariant measures},
  journal = {J. Soviet Math.},
  volume  = {9},
  pages   = {9--18},
  year    = {1978}
}

@article{Borell,
  title={The brunn-minkowski inequality in gauss space},
  author={Borell, Christer},
  journal={Inventiones mathematicae},
  volume={30},
  number={2},
  pages={207--216},
  year={1975},
  publisher={Springer-Verlag Berlin/Heidelberg}
}

@article{BaiLuXia,
  author = {J. Bai and C. Lu and C. Xia},
  title  = {Capillary isoperimetric inequality in a Gaussian half-space},
  note   = {submitted}
}

@article{CongedoTamanini,
  author  = {I. Tamanini and G. Congedo},
  title   = {Density theorems for local minimizers of area-type functionals},
  journal = {Rend. Sem. Mat. Univ. Padova},
  volume  = {85},
  pages   = {217--248},
  year    = {1991}
}

@article{Souam,
  title={On Stable Capillary Hypersurfaces with Planar Boundaries: R. Souam},
  author={Souam, Rabah},
  journal={The Journal of Geometric Analysis},
  volume={33},
  number={6},
  pages={196},
  year={2023},
  publisher={Springer}
}

\end{document}